\documentclass{article}

\usepackage[utf8]{inputenc}
\usepackage{amsmath}
\usepackage{amssymb}
\usepackage{yfonts}
\usepackage{faktor}
\usepackage{amsfonts}
\usepackage{amsthm}
\usepackage{bbding}
\usepackage{bm} 
\usepackage{graphicx}
\usepackage{fancyvrb}
\usepackage{xargs}

\title{Models of Bounded Arithmetic and variants of Pigeonhole Principle}
\author{Mykyta Narusevych\footnote{Supported by the Charles University project PRIMUS/21/SCI/014, Charles University Research Centre program No. UNCE/24/SCI/022 and GA UK project No. 246223}}
\date{Faculty of Mathematics and Physics \\
Charles University\footnote{Sokolovská 83, Prague, 186 75, The Czech Republic}}

\begin{document}

\newcommand{\TR}{T^1_2(R)}
\newcommand{\DR}{\Delta^b_1(R)}
\newcommand{\PHPR}[2]{PHP^{#1}_{#2}(R)}
\newcommand{\PHPD}[2]{PHP^{#1}_{#2}(\DR)}
\newcommand{\OPHPR}[1][n]{ontoPHP^{#1+1}_{#1}(R)}
\newcommandx{\WPHPR}[2][1=m, 2=R]{WPHP^{2#1}_{#1}(#2)}
\newcommand{\WPHPD}[1][m]{WPHP^{2#1}_{#1}(\DR)}

\newcommand{\M}{\mathbb{M}}
\newcommand{\I}{\mathbb{I}}
\renewcommand{\P}{\mathbb{P}}
\newcommand{\Pdef}{dense and $\P$-definable }

\newcommand{\WAs}[2]{$A^{#2}_{#1, \sigma}$-array }
\newcommand{\WAt}[2]{$A^{#2}_{#1, \tau}$-array }
\newcommand{\WAr}[2]{$A^{#2}_{#1, \rho}$-array }

\newcommand{\PHPTB}[2]{PHP^{#1}_{#2}\text{-tree} }
\newcommand{\PHPTU}[1]{PHP^{#1}\text{-tree} }
\newcommand{\PHPT}{PHP\text{-tree} }

\theoremstyle{definition}
\newtheorem{theorem}{Theorem}[section]
\newtheorem{definition}[theorem]{Definition}
\newtheorem{lemma}[theorem]{Lemma}
\newtheorem{corollary}[theorem]{Corollary}
\newtheorem{proposition}[theorem]{Proposition}
\newtheorem*{remark}{Remark}
\newtheorem*{conjecture}{Conjecture}
\newtheorem*{problem}{Problem}
\newtheorem*{fact}{Fact}

\maketitle

\begin{abstract}
    We give elementary proof that theory $\TR$ augmented by the weak pigeonhole principle for all $\DR$-definable relations does not prove the bijective pigeonhole principle for $R$. This can be derived from known more general results but our proof yields a model of $\TR$ in which $\OPHPR$ fails for some nonstandard element $n$ while $PHP^{m+1}_m$ holds for all $\DR$-definable relations and all $m \leq n^{1-\epsilon}$, where $\epsilon > 0$ is a fixed standard rational parameter. This can be seen as a step towards solving an open question posed by M. Ajtai in \cite[Page~3]{ajtai90}.
\end{abstract}

\section{Introduction}

This paper studies first-order bounded arithmetic theories in a language extended by an unspecified relation symbol $R$ (such theories are often called \textit{relativized}). The main theories we are to inspect all contain theory $\TR$, a subtheory of a theory $T_2(R)$, and are augmented by instances of (variants of) the pigeonhole principle. Several separations among bounded arithmetic theories and unprovability of various combinatorial principles are known (see \cite[Chapter~10]{krajicek95}). What is not known is whether the consecutive theories $T^i_2(R)$ and $T^{i+1}_2(R)$ (subtheories of $T_{2}(R)$ with induction restricted to $\Sigma^b_i(R)$-formulas and $\Sigma^b_{i+1}(R)$-formulas, respectively) can be separated by principles of a fixed quantifier complexity independent of $i$. This problem relates to the depth $d$ versus $d+1$ problem in proof complexity (see \cite[Problem 14.3.2]{krajicek19}). It is only known that if we weaken the theories by excluding the smash function $x\#y$ (the resulting theories are denoted $T^i_1(R)$) such a separation is possible (see \cite{impagliazzokrajicek02}).

We think that to make progress on this problem it may be useful to have stand-alone and elementary proofs of various separation results that can be obtained by more general methods aiming at a large variety of stronger theories. For example, the random restriction method using the $PHP$ switching lemma \cite{ajtai88} applies to the whole theory $T_2(R)$ and in a sense blurs the distinction between its fragments. 

The main technical result is a new and elementary proof of the following statement:
\begin{align}
\label{main result}
    \TR + \WPHPD \nvdash \OPHPR,
\end{align}
where $\OPHPR$ denotes the bijective pigeonhole principle stated for a binary relation symbol $R$ and $\WPHPD$ denotes the weak pigeonhole principle stated for all $\DR$-definable relations.

The qualification ``elementary" means that the proof does not refer to any
general frameworks aiming at a larger class of theories or a too-strong
method aiming at stronger theories. For example, we can note that (\ref{main result}) can be shown indirectly by the following argument. Firstly, it is known that $T_2(R)$ (and even $T^2_2(R)$) proves $\WPHPR$ (see \cite{pariswilkiewoods88}). This can be used to derive $T_2(R) \vdash \WPHPD)$ (the same is true even for $T^2_2(R)$). Secondly, using the restrictions technique and applying the already mentioned $PHP$ switching lemma one can show that $T_2(R) \nvdash \OPHPR$ (see \cite{ajtai88}, \cite{krajicekpudlakwoods95} and \cite{pitassibeameimpagliazzo95}). Combining these results we can derive the mentioned statement. Another general method is that of \cite{muller20} which extends an original construction of \cite{riis93}. The statement (\ref{main result}) is then an instance of a more general theorem showing, in short, that over $\TR$ any $weak$ combinatorial principle cannot prove a $strong$ one (see \cite{muller20} for details).

Our method has the advantage that its simple variant can be used
to construct a model of:
\begin{align}
\label{secondary result}
    \TR + \exists n (\forall m \leq n^{1-\epsilon} \PHPD{m+1}{m} + \neg \OPHPR),
\end{align}
with $\epsilon > 0$ being an arbitrarily small standard rational parameter. Neither of the two mentioned general
methods that can be used to prove (\ref{main result}) seem to be modifiable to prove (\ref{secondary result}). 

As mentioned in the abstract, (\ref{secondary result}) can be seen as a step towards an open question posed by M. Ajtai in \cite[Page~3]{ajtai90}. The full problem can be stated as follows.
\begin{problem}[M. Ajtai \cite{ajtai90}]
\label{ajtaiProblem}
Does the theory:
\begin{align*}
    T_{2}(R) + \forall m \leq n/2 (PHP_{m}^{m + 1}(\Sigma_{\infty}^{b}(R)))
\end{align*}
prove $PHP_{n}^{n + 1}(R)$?
\end{problem}

\section{Preliminaries}

The family of theories we are studying in this paper are all sound axiomatizations of true arithmetic over a first-order language denoted simply as $L$. This language consists of constant symbols $0$ and $1$, unary operations $\lfloor\frac{x}{2}\rfloor$ and $|x|$, binary operations $x+y$, $x \cdot y$ and $x \# y$, and binary relation $x < y$. The intended meanings of $0, 1, \lfloor\frac{x}{2}\rfloor, x+y, x \cdot y, x < y$ are obvious. Symbol $|x|$ denotes bit-length of $x$, i.e. $|x| = \lfloor \log_{2} x \rfloor + 1$, and $x \# y$ stands for binary smash function which is standardly interpreted as $2^{|x|\cdot |y|}$.

The main theory for $L$ is the \textit{bounded arithmetic} $T_{2}$ which is axiomatized by a finite universal theory $BASIC$ together with the induction scheme for all \textit{bounded formulas}. Recall that a formula is said to be bounded if all quantifiers occurring in it are of the form $\forall x < t$ or $\exists x < t$ with an $L$-term $t$ not containing $x$ (see \cite{buss85} or \cite{krajicek19} for a comprehensive treatment of $T_2$ and its fragments).

To define theories $S^1_2$ and $T^1_2$ and $\Delta^b_1$-formulas we need to recall the class $\Sigma^b_1$ of bounded $L$-formulas. These are formulas formed from atomic ones using connectives, any \textit{sharply bounded} quantifiers, and bounded existential quantifiers at the front. The class $\Pi^b_1$ is defined dually using bounded universal quantifiers. Sharply bounded quantifiers are of the form $\forall x < |t|$ and $\exists x < |t|$ with an $L$-term $t$ not containing $x$.

Theory $S^1_2$ extends the already mentioned $BASIC$ by the scheme of polynomial induction $PIND$ of the form:
\begin{align*}
    (\phi(0) \land \forall x (\phi(\lfloor \frac{x}{2} \rfloor) \to \phi(x))) \to \forall x \phi(x),
\end{align*}
for all $\Sigma^b_1$-formulas, while $T^1_2$ has ordinary induction $IND$ for all $\Sigma^b_1$-formulas. The induction formulas may contain, as is usual, other free variables than the induction variable.

A $\Sigma^b_1$-formula $\phi(x)$ is said to be $\Delta^b_1$ in $S^1_2$ iff there is a $\Pi^b_1$-formula $\psi(x)$ so that $S^1_2 \vdash \phi(x) \equiv \psi(x)$. We often leave out the reference ``in $S^1_2$" and consider it automatically included. It can be shown that $\Delta^b_1$-definable sets on $\mathbb{N}$ are exactly those decidable by polynomial-time algorithms \cite[page 64, Theorem 2]{buss85}.

Already the weakest of the mentioned theories $S^1_2$ is strong enough to talk about bounded sets or sequences, i.e. this theory is \textit{sequential} (consult \cite[Lemma~5.1.5]{krajicek95}).

This paper focuses on the unprovability of a particular combinatorial statement in a \textit{relativized} version of $T^1_2$ denoted as $\TR$ (and in a particular extension of the mentioned theory). To obtain such a theory we first expand the language $L$ by a single binary relation $R(x, y)$. One can naturally define classes $\Sigma^b_1(R)$ and $\Pi^b_1(R)$. $\TR$ is axiomatized by $BASIC$ together with the induction scheme for all $\Sigma^b_1(R)$-formulas. Theories $S^1_2(R), T_2(R)$ and the class $\Delta^b_1(R)$ in $S^1_2(R)$ are defined analogously. Similarly as before, one can show that $\Delta^b_1(R)$-definable sets on $\mathbb{N}$ are exactly those decidable by polynomial time algorithms with oracle access to $R$.

It is a well-known fact that the \textit{least number principle} axiom scheme is equivalent to induction when stated for all formulas. The same is true for $\Sigma^b_1$-formulas, i.e. theories $LNP(\Sigma_{1}^{b})$ and $T^1_2$ are equivalent (see \cite[Lemma~5.2.7]{krajicek95}). The situation is the same in the relativized case.

\begin{definition}
\label{pigeonhole principle}
For parameters $n, m$ we define $\PHPR{n}{m}$ as a disjunction of the following formulas:
\begin{itemize}
    \item $\exists a < n \: \forall b < m : \neg R(a, b)$,
    \item $\exists a \neq a' < n \: \exists b < m : R(a, b) \land R(a', b)$,
    \item $\exists a < n \: \exists b \neq b' < m : R(a, b) \land R(a, b')$.
\end{itemize}
We define $ontoPHP_{n}^{m}(R)$ as a disjunction of $\PHPR{n}{m}$ with the formula:
\begin{itemize}
    \item $\exists b < m \: \forall a < n: \neg R(a, b)$.
\end{itemize}
We denote $\PHPR{2m}{m}$ as $\WPHPR$ for cosmetic reasons. $\PHPD{n}{m}$ stands for a set of instances of the above disjunction with $R$ substituted for all $\DR$-formulas. As usual, those formulas may contain additional free variables than $a$ and $b$. Similarly, we denote $\PHPD{2m}{m}$ as $\WPHPD$.
\end{definition}

\section{Forcing}

We show the statement (\ref{main result}) by constructing a model of:
\begin{align*}
    \TR + \forall m \WPHPD + \exists n \neg \OPHPR,
\end{align*}
employing forcing. 

Forcing is a generic name for a class of methods and its various formal variants for bounded arithmetic were developed, cf. \cite{pariswilkie85, ajtai88, krajicek95, atseriasmuller15, riis93, muller20}. We want to keep the presentation elementary and as self-contained as possible and hence we do not refer to any of the general expositions. Our approach is pretty much the same as the one in \cite{pariswilkie85} or \cite{riis93}.

Throughout the following two sections, $\M$ is a countable non-standard model of $Th(\mathbb{N})$ in the $L$-language, and $n$ is a non-standard number from $\M$. We denote the set $\{m \in \M \: | \: m < n\}$ as $[n]$. 

We let $\I$ be a substructure of $\M$ with the domain $I$ defined as $\{m \: | \: m < 2^{|n|^{c}} \text{for some standard } c\}$. $\I$ will eventually play the role of the universe of a structure where $R$ will be suitably interpreted, although all the arguments through the following sections are done inside $\M$. The main reason behind choosing such $\I$ is that numbers of magnitude $2^n$ are absent from $\I$, while they do exist in $\M$.

Since all the formulas representing various combinatorial principles discussed in the previous section may contain free variables, whenever we mention $L(R)$-formulas or $L(R)$-sentences in the current context we allow them to contain parameters from $\I$.

\begin{definition}
\label{conditions}
Let $\P$ be the set of all partial injective functions in $\M$ of size $\leq |n|^c$ for standard $c$ between the sets $[n + 1]$ and $[n]$. We call members of $\P$ \textit{conditions} and we assume $\P$ is partially ordered by inclusion. The relation $ \sigma \supseteq \rho$ is also denoted $ \sigma \leq \rho$. 

We say that $\sigma$ and $\tau$ from $\P$ are \textit{compatible} (denoted $\sigma \| \tau$), if there is $\delta \in \P$ so that $\delta \leq \sigma$ and $\delta \leq \tau$, otherwise $\sigma$ and $\tau$ are called \textit{incompatible} (denoted $\sigma \perp \tau$).

We say that a set $D \subseteq \P$ is $\P$-\textit{definable}, if it is definable in $(\M, \P)$, the expansion of $\M$ by predicate for $\P$. We say that $D \subseteq \P$ is \textit{dense}, if for all $\sigma \in \P$ there is $\tau \in D$ so that $\tau \leq \sigma$.
\end{definition}

Note that conditions $\sigma$ and $\tau$ are compatible iff $\sigma \cup \tau \in \mathbb{P}$.

\begin{definition}
\label{generic}
A $G \subseteq \P$ is a \textit{filter}, if all the conditions from $G$ are pairwise compatible and for any $\sigma \in G$ and $\tau \in \P$ so that $\tau \geq \sigma$ it holds that $\tau \in G$ (i.e. $G$ is \textit{upwards closed}). 

We say that a filter $G$ is \textit{generic} if for any \Pdef $D$ it holds that $G \cap D \neq \emptyset$.
\end{definition}

The following statement is well-known.

\begin{proposition}
\label{generic filter exists}
For any $\sigma \in \P$ there exists a generic filter $G$ containing $\sigma$.
\end{proposition}

\begin{proposition}
\label{generic bijection}
Let $G$ be any generic filter. Then $\bigcup G$ is a bijective function between the sets $[n + 1]$ and $[n]$. We will denote such function as $R_{G}$.
\end{proposition}
\begin{proof}
It is clear that for any filter $F$ the set $\bigcup F$ is a partial injective function between the sets $[n + 1]$ and $[n]$. To show the totality of $R_{G}$ let $a \in [n + 1]$ and let $D_{a}$ be the set of all conditions defined on $a$. Such $D_{a}$ is \Pdef and so $R_{G}$ is defined on $a$. One then shows the surjectivity of $R_{G}$ in a similar way by considering the set $D^{b}$ of all conditions having $b \in [n]$ in their range.
\end{proof}

For $R_{G}$ as above we denote the $L(R)$-structure with the domain $\I$ where $R_{G}$ interprets $R$ as $(\I, R_{G})$.

\begin{definition}
\label{forcing}
Let $\sigma \in \P$ and $\phi$ be an $L(R)$-sentence. We say that $\sigma$ \textit{forces} $\phi$ or $\phi$ is \textit{forced by} $\sigma$ (denoted as $\sigma \Vdash \phi$) if for any generic filter $G$ containing $\sigma$ it holds that $(\I, R_{G}) \vDash \phi$.
\end{definition}

Note that if sigma forces $\phi$ and $\tau \leq \sigma$, then $\tau$ forces $\phi$ too.

\begin{corollary}
\label{forcing negation of bijective pigeonhole principle}
Any condition $\sigma$ (or, equivalently, the empty condition $\sigma = \emptyset$) forces $\neg \OPHPR$.
\end{corollary}

We now discuss the forcing of various types of formulas. What is needed for the main results can be summarized as follows.

\begin{theorem}
\label{forcing properties}
Let $\sigma \in \mathbb{P}$. Then the following holds:
\begin{enumerate}
    \item for $a \in [n + 1]$ and $b \in [n]$ it holds that $\sigma \Vdash R(a, b)$ iff $\{(a, b)\} \subseteq \sigma$;
    \item for $a \in [n + 1]$ and $b \in [n]$ it holds that $\sigma \Vdash \neg R(a, b)$ iff $\{(a, b)\} \perp \sigma$;
    \item for any two $L(R)$-sentences $\phi$ and $\theta$ it holds that $\sigma \Vdash \phi \land \theta$ iff $\sigma \Vdash \phi$ and $\sigma \Vdash \theta$;
    \item for any $L(R)$-formula $\phi(x)$ it holds that $\sigma \Vdash \forall x \: \phi(x)$ iff $\forall a \in \I \: : \: \sigma \Vdash \phi(a)$;
    \item for any sharply bounded $L(R)$-sentence $\phi$ it holds that $\sigma \Vdash \neg \phi$ iff $\forall \tau \leq \sigma$ it holds that $\tau \nVdash \phi$.
    \item for any two sharply bounded $L(R)$-sentences $\phi$ and $\theta$ it holds that $\sigma$ forces $\phi \lor \theta$ if and only if $\forall \tau \leq \sigma \: \exists \rho \leq \tau \: : \: \rho \Vdash \phi$ or $\rho \Vdash \theta$;
    \item for sharply bounded formula $\phi(\overline{x})$ it holds that $\sigma$ forces $\exists \overline{x} \: \phi(\overline{x})$ if and only if $\forall \tau \leq \sigma \: \exists \rho \leq \tau \: \exists \overline{a} \in \I \: : \: \rho \Vdash \phi(\overline{a})$.
\end{enumerate}
\end{theorem}
We split the proof of the above result into separate lemmas.

\begin{definition}
\label{forcing equivalence}
We say that $L(R)$-sentences $\phi$ and $\theta$ are \textit{forcing-equivalent} (or, simply \textit{f-equivalent}) if for any $\sigma \in \P$ it holds that $\sigma \Vdash \phi$ iff $\sigma \Vdash \theta$.
\end{definition}

Logical equivalence implies f-equivalence. So we may safely assume all the $L(R)$-formulas we are discussing are in prenex normal forms. Furthermore, evaluating all the $L$-subformulas and all the $L$-subterms in $\M$ of some $L(R)$-sentence leads to an f-equivalent sentence and so we may further assume that all the atomic subformulas are of the form $R(a, b)$ or $\neg R(a, b)$ for some $a, b \in \M$ (actually in $\I$ since we allow parameters from $\I$ only).

\begin{lemma}
\label{forcing atomic formulas}
Let $\sigma \in \P$, $a \in [n + 1]$ and $b \in [n]$. Then, $\sigma \Vdash R(a, b)$ iff $\{(a, b)\} \subseteq \sigma$ and $\sigma \Vdash \neg R(a, b)$ iff $\{(a, b)\} \perp \sigma$.
\end{lemma}
\begin{proof}
Assuming $\{(a, b)\} \subseteq \sigma$ it is clear that $\sigma \Vdash R(a, b)$. So assume that $\{(a, b)\} \nsubseteq \sigma$. We can extend $\sigma$ to $\tau$ so that $(a', b) \in \tau$ or $(a, b') \in \tau$ for $a' \neq a$ and $b' \neq b$. Such $\tau$ forces $\neg R(a, b)$ implying $\sigma \nVdash R(a, b)$. 

The equivalence for $\neg R(a, b)$ is done similarly. 
\end{proof}

\begin{lemma}
\label{forcing conjunction and universal quantifier}
Let $\sigma \in \P$. For any two $L(R)$-sentences $\phi$ and $\theta$ it holds that $\sigma \Vdash \phi \land \theta$ iff $\sigma \Vdash \phi$ and $\sigma \Vdash \theta$. For any $L(R)$-formula $\psi(x)$ it holds that $\sigma \Vdash \forall x \psi(x)$ iff for all $a \in \I : \sigma \Vdash \psi(a)$.  
\end{lemma}
\begin{proof}
Both equivalences follow immediately from a more general statement that $\sigma \Vdash \bigwedge_i \phi_i$ iff $\forall i \: \sigma \Vdash \phi_i$. The statement itself easily follows from the definitions.
\end{proof}

\begin{lemma}
\label{forcing negation}
Let $\sigma \in \P$ and $\phi$ be a sharply bounded $L(R)$-sentence. Then $\sigma \Vdash \neg \phi$ iff $\forall \tau \leq \sigma$ it holds that $\tau \nVdash \phi$.
\end{lemma}
\begin{proof}
The left-to-right direction is clear. Assume now that $\sigma \nVdash \neg \phi$. The statement then follows from an auxiliary observation: given a structure $(\I, R_G)$ satisfying $\phi$ one can find for a suitable standard $c$ a $|n|^c$-large subset $\tau$ in $\M$ of $R_{G}$ which forces $\phi$. This, in turn, can be established by induction on logical complexity of $\phi$. Below we show a sketch of the proof.

For atomic or negated atomic $\phi$ the claim follows from \ref{forcing atomic formulas}. Cases for conjunction and disjunction are shown in the exact same way as the ones for sharply bounded quantifiers below.

Let $\phi$ be of the form $\forall x \leq |n|^c \: \psi(x)$. By induction hypothesis we find $\tau_x$ for each $x \leq |n|^c$ so that $\tau_x \Vdash \psi(x)$. Note that all $\tau_x$ are pairwise compatible and so we can define $\tau = \bigcup_x \tau_x$ which satisfies $|\tau| \leq |n|^{c'}$ for a suitable $c'$ and $\tau \Vdash \forall x \leq |n|^c \psi(x)$.

Finally, let $\phi$ be of the form $\exists x \leq |n|^c \: \psi(x)$. Similarly as above, we pick particular $x$ so that $\tau_x \Vdash \psi(x)$ and let $\tau = \tau_x$.
\end{proof}

\begin{lemma}
\label{forcing disjunction}
Let $\sigma \in \P$. For any two sharply bounded $L(R)$-sentences $\phi$ and $\theta$ it holds that $\sigma \Vdash \phi \lor \theta$ iff $\forall \tau \leq \sigma \: \exists \rho \leq \tau \: : \: \rho \Vdash \phi$ or $\rho \Vdash \theta$.
\end{lemma}
\begin{proof}
Let $\sigma$ force $\phi \lor \theta$ and let $\tau \leq \sigma$. Assume for contradiction that $\forall \rho \leq \tau$ $\rho \nVdash \phi$ and $\rho \nVdash \theta$. Since both $\phi$ and $\theta$ are sharply bounded we use Lemma \ref{forcing negation} to derive $\tau \Vdash \neg \phi$ and $\tau \Vdash  \neg \theta$ which is equivalent to $\tau \Vdash \neg \phi \land \neg \theta$ by Lemma \ref{forcing conjunction and universal quantifier}. This, however, contradicts the fact that $\sigma \Vdash \phi \lor \theta$ and $\tau \leq \sigma$.

Now assume $\sigma \nVdash \phi \lor \theta$. Note that Lemma \ref{forcing negation} can be formulated as $\pi \nVdash \psi$ iff $\exists \delta \leq \pi$ so that $\delta \Vdash \neg \psi$ for $\psi$ a sharply bounded formula. Since $\phi \lor \theta$ is sharply bounded, we can use the above statement to extend $\sigma$ to $\tau$ so that $\tau \Vdash \neg \phi \land \neg \theta$. Such $\tau$ can not be extended to $\rho$ forcing $\phi$ or $\theta$.
\end{proof}

Before we prove the last part of Theorem \ref{forcing properties} we state the following proposition which is readily established by induction on logical complexity together with the previously proved lemmas.

\begin{proposition}
\label{definability forcing}
Let $\phi(\overline{x})$ be a sharply bounded formula. The set of all pairs $(\sigma, \overline{c})$ so that $\sigma \Vdash \phi(\overline{c})$ is $\P$-definable.
\end{proposition}

\begin{lemma}
\label{forcing existential quantifier}
For any sharply bounded $L(R)$-formula $\phi(\overline{x})$ it holds that $\sigma \Vdash \exists \overline{x} \psi(\overline{x})$ iff $\forall \tau \leq \sigma \: \exists \rho \leq \tau \: \exists \overline{a} \in \I : \: \rho \Vdash \phi(\overline{a})$.
\end{lemma}
\begin{proof}
Assume $\sigma \Vdash \exists \overline{x} \phi(\overline{x})$ and let $\tau \leq \sigma$. Note that $\tau \Vdash \exists \overline{x} \phi(\overline{x})$, too. Let $G$ be a generic filter containing $\tau$ such that it holds $(\I, R_{G}) \vDash \exists \overline{x} \phi(\overline{x})$. We can then find $\overline{a}$ so that $(\I, R_{G}) \vDash \phi(\overline{a})$ and as in the proof of Lemma \ref{forcing negation} we can take a $|n|^c$-large subset $\rho$ in $\M$ of $R_{G}$ so that $\rho \leq \tau$ and $\rho \Vdash \phi(\overline{a})$.

Conversely, assume $\forall \tau \leq \sigma \: \exists \rho \leq \tau \: \exists \overline{a} \in \I : \: \rho \Vdash \phi(\overline{a})$. Note that this means the set $D$ of all conditions $\rho$ for which there is $\overline{a}$ so that $\rho \Vdash \phi(\overline{a})$ is \textit{dense relative} to $\sigma$, i.e. the density condition is satisfied for all $\tau$ extending $\sigma$.  

By Proposition \ref{definability forcing} it follows that the set $D$ is also $\P$-definable. Thus any generic filter $G$ containing $\sigma$ must contain some condition $\rho$ forcing $\phi(\overline{a})$ for some $\overline{a}$ proving $\sigma \Vdash \exists \overline{x} \phi(\overline{x})$. The latter is true since any generic filter containing $\sigma$ must intersect any $\P$-definable set $D$ which is dense relative to $\sigma$. This can be shown by considering the set $D' = D \cup \{\rho \: | \: \rho \perp \sigma\}$ which is \Pdef and so for any generic $G$ there is some $\rho \in G \cap D'$. But if $G$ contains $\sigma$, such $\rho$ must be compatible with $\sigma$ and so is necessary a member of $D$. 
\end{proof}

The following theorem is a fundamental fact regarding $\TR$ and forcing construction. The first version of the theorem appeared in \cite{pariswilkie85} for closely related theory $I\exists_1(R)$ and its modification for $\TR$ was shown in \cite{riis93}. We skip the proof, the reader can find details in \cite[pp. 273 - 274]{krajicek95}.

\begin{theorem}
\label{Paris, Wilkie}
For any sharply bounded $L(R)$-sentence $\phi(x, \overline{y})$ it holds that:
\begin{align*}
    \emptyset \Vdash LNP(\exists \overline{y} \phi(x, \overline{y})).
\end{align*}
\end{theorem}

\begin{corollary}
\label{generic model}
For any generic filter $G$ it holds that:
\begin{align*}
    (\I, R_{G}) \vDash \TR + \exists n \neg \OPHPR.
\end{align*}
\end{corollary}

\section{\textit{WPHP}-Arrays}
\label{section WPHP arrays}

This whole section is devoted to the proof of:
\begin{align*}
    \emptyset \Vdash \WPHPD,
\end{align*}
where $m \in \I$.

The argument is as follows: assuming $\emptyset \nVdash WPHP^{2m}_m(\phi(x, y))$ for a $\DR$-formula $\phi(x, y)$ we construct a certain combinatorial structure. Then, by using two different arguments we establish upper and lower bounds on the size of such structure resulting in a contradiction, thus showing that $\emptyset \Vdash \WPHPD$. 

The following concepts are defined in the ground model $\M$.

\begin{definition}
\label{WPHP array}
Let $\sigma \in \P$ and $m, k \in \M$ with $k \leq |n|^c$ for standard $c$. We define an \textit{$(m,k,\sigma)$-WPHP-array} (or just \WAs{m}{k}) as an array $A$ indexed by $[2m]\times[m]$ so that each entry is a set of conditions of sizes bounded from above by $k$ and satisfying the following properties:
\begin{enumerate}
    \item $\forall a \neq a' \in [2m] \: \forall b \neq b' \in [m] \: : \: A(a, b) \cap A(a', b) = A(a, b) \cap A(a, b') = \emptyset$;
    \item $\forall a, a' \in [2m] \: \forall b \in [m] \: \forall \tau \in A(a, b) \: \forall \tau' \in A(a', b) \: : \: \tau \neq \tau' \to \tau \perp \tau'$;
    \item $\forall a \in [2m] \: \forall b, b' \in [m] \: \forall \tau \in A(a, b) \: \forall \tau' \in A(a, b') \: : \: \tau \neq \tau' \to \tau \perp \tau'$;
    \item $\forall \rho \leq \sigma \: \forall a \in [2m] \: \exists b \in [m] \: \exists \tau \in A(a, b) \: : \: \tau \| \rho$.
\end{enumerate}
We require that each $\tau \in A(a, b)$ is compatible with $\sigma$ and does not intersect $\sigma$.
\end{definition}

\begin{proposition}
\label{size of WPHP array}
Let $A$ be \WAs{m}{k}. Define the \textit{size} of $A$ as the number $N = \sum_{a, b}|A(a, b)|$, where $|A(a, b)|$ is a number of conditions in $A(a, b)$. Then:
\begin{align*}
    N = \sum_{a}|\bigcup_{b}A(a, b)| = \sum_{b}|\bigcup_{a}A(a, b)|.
\end{align*}
\end{proposition}
\begin{proof}
This follows from the first property of $A_{m, \sigma}^{k}$-arrays.
\end{proof}

It is hard to compute non-trivial lower and upper bounds on the size of a general \WAs{m}{k} $A$. The way to proceed is to create a related array $A'$ with nicer combinatorial properties.
\begin{definition}
An \WAs{m}{k} $A$ is called $k'$-\textit{uniform} (or just \textit{uniform} in case $k' = k$), if all the conditions from the set $\bigcup_{a, b}A(a, b)$ are of the size $k'$.
\end{definition}
As we will show later, uniformity alone is enough to compute non-trivial upper bounds on the size of a \WAs{m}{k}. However, it is not enough to calculate meaningful lower bounds.

\begin{definition}
\label{PHP tree}
Let $D \subseteq [n + 1]$ and $R \subseteq [n]$. A \textit{PHP-tree} over $D$ and $R$ is defined inductively as follows:
\begin{itemize}
    \item a single node (a root) is a $\PHPT$ over any $D$ and $R$;
    \item for every $a \in D$ the following is a $PHP$-tree over $D$ and $R$:
    \begin{itemize}
        \item at the root the tree branches according to all $b \in R$, labeling the corresponding edge $(a, b)$;
        \item at the end-point of the edge labeled by $(a, b)$ the tree continues as a $\PHPT$ over $D \setminus \{a\}$ and $R \setminus \{b\}$;
    \end{itemize}
    \item for every $b \in R$ the following is a $\PHPT$ over $D$ and $R$:
    \begin{itemize}
        \item at the root the tree branches according to all $a \in D$, labeling the corresponding edge $(a, b)$;
        \item at the end-point of the edge labeled by $(a, b)$ the tree continues as a $\PHPT$ over $D \setminus \{a\}$ and $R \setminus \{b\}$.
    \end{itemize}
\end{itemize}
We further label each node of a $\PHPT$ by the set of pairs labeling edges leading to the given node from the root (the root itself is labeled as $\emptyset$).

For $\sigma \in \P$ we define $\PHPTU{\sigma}$ as a $\PHPT$ over $[n + 1] \setminus D_{\sigma}$ and $[n] \setminus R_{\sigma}$, where $D_{\sigma}$ is the domain of $\sigma$ and $R_{\sigma}$ is the range of $\sigma$. 

For $\sigma, \tau \in \P$ which are compatible we define $\PHPTB{\sigma}{\tau}$ as a $\PHPTU{\sigma}$ with the property that for each $(a, b) \in \tau \setminus \sigma$ every leaf in the given tree is labeled by a set either containing $(a, b)$ or containing $(a', b)$ and $(a, b')$ for $a \neq a'$ and $b \neq b'$.

We say that a $\PHPT$ is $k$-\textit{uniform} if all its maximal paths (starting from the root) have the same length equal to $k$.
\end{definition}
For a detailed discussion about $PHP$-trees and their applications consult \cite[Chapter~15.1]{krajicek19}. 

Since each $\sigma \in \P$ is an element of $\M$ we can talk about the size of $\sigma$ (denoted as $|\sigma|$) which may be non-standard.

\begin{proposition}
\label{size of uniform PHP tree}
Let $\sigma \in \P$ and $P$ a $k$-uniform $\PHPTU{\sigma}$-tree. The number of maximal paths (i.e. the number of leaves) in $P$ is lower-bounded by $\frac{(n - |\sigma|)!}{(n - |\sigma| - k)!}$.
\end{proposition}
\begin{proof}
The minimal number of paths is achieved for a $\PHPTU{\sigma}$ with nodes labeled by elements of the larger set (initially $[n + 1] \setminus D_{\sigma}$ of size $n + 1 - |\sigma|$) and branching according to elements from the smaller set (initially $[n] \setminus R_{\sigma}$ which has size $n - |\sigma|$).
\end{proof}

\begin{definition}
\label{PHP tree conditions}
We say that a set of conditions $T \in \M$ is \textit{realized by a PHP-tree} if there is a $\PHPT$ $P$ so that $T$ equals the set of labels of leaves of $P$. Similarly, we define realizations by $\PHPTU{\sigma}$, $\PHPTB{\sigma}{\tau}$ and uniform versions of the previous notions.
\end{definition}

\begin{corollary}
\label{size of uniform PHP tree conditions}
If $T \in \M$ is realized by a $k$-uniform $\PHPTU{\sigma}$, then $|T| \geq \frac{(n - |\sigma|)!}{(n - |\sigma| - k)!}$.
\end{corollary}

\begin{proposition}
\label{PHP tree conditions properties}
Let $T \in \M$ be realized by a $\PHPTU{\sigma}$. It then holds:
\begin{enumerate}
    \item $\forall \tau, \tau' \in T \: : \: \tau \neq \tau' \to \tau \perp \tau'$;
    \item $\forall \rho \leq \sigma \: \exists \tau \in T \: : \: \tau \| \rho$.
\end{enumerate}
\end{proposition}
\begin{proof}
The first property follows from the fact that any two labelings of the leaves in any $\PHPT$ are incompatible. The second property follows from the definitions.
\end{proof}

We now state a lemma which allows us to \textit{glue} together various $PHP$-trees to get a bigger one. The idea is rather simple, every $\PHPT$ can be extended to a deeper one by simply appending another $\PHPT$ to a leaf in the original tree, assuming that the tree we are inserting is compatible with the path leading to such a leaf. Of course, one needs to ensure that the tree we get is definable in $\M$. An immediate corollary is that every $\PHPT$ can be extended to a $k$-uniform one for a suitable $k$.

\begin{lemma}
\label{PHP trees refinement}
Let $P$ be a $\PHPTB{\sigma}{\tau}$. Let $(S_{\lambda})_{\lambda} \in \M$ a collection of $PHP$-trees where $\lambda$ ranges over all labels of the leaves of $P$. For any such $\lambda$ we assume $S_{\lambda}$ is a $\PHPTU{\sigma \cup \lambda}$. 

We define a tree $P \triangleleft S$ as an extension of $P$ by $(S_{\lambda})_{\lambda}$ where we append $S_{\lambda}$ to the leaf of $P$ labeled by $\lambda$. It 

Then, $P \triangleleft S$ is a $\PHPTB{\sigma}{\tau}$ with the property that each maximal path of $P \triangleleft S$ extends a unique maximal path of $P$.
\end{lemma}
\begin{proof}
$P \triangleleft S$ is definable in $\M$, since both $P$ and $(S_{\lambda})_\lambda$ are in $\M$. 

The fact that each $S_{\lambda}$ is a $\PHPTU{\sigma \cup \lambda}$ implies $P \triangleleft S$ is a $\PHPT$. 

Since $P$ is a $\PHPT$ over $[n + 1] \setminus D_{\sigma}$ and $[n] \setminus R_{\sigma}$ and $S_{\lambda}$ is a $\PHPT$ over $[n + 1] \setminus D_{\sigma \cup \lambda}$ and $[n] \setminus R_{\sigma \cup \lambda}$ it follows that $P \triangleleft S$ is $\PHPT$ over $[n + 1] \setminus D_{\sigma}$ and $[n] \setminus R_{\sigma}$ and thus is a $\PHPTU{\sigma}$. 

Each maximal path of $P \triangleleft S$ extends a unique maximal path of $P$ by the definition of extension. 

Finally, since any label $\pi \in \P$ of a leaf of $P \triangleleft S$ is a superset of the label of a leaf of $P$ it follows that $\pi$ does contain $(a, b)$ or $(a', b)$ and $(a, b')$ with $a \neq a'$ and $b \neq b'$ for all $(a, b) \in \tau \setminus \sigma$. This implies $P \triangleleft S$ is a $\PHPTB{\sigma}{\tau}$.
\end{proof}

In general, one can consider appending $PHP$-trees to leaves of another $\PHPT$ the above assumptions, namely that each $S_{\lambda}$ is a $\PHPTU{\lambda}$ (we are ignoring the initial condition $\sigma$ here). In such a case one needs to ensure that the resulting tree is a $\PHPT$ by first removing all branches of $S_{\lambda}$ which are incompatible with $\lambda$. We then shrink all nodes of the resulting tree which have only a single edge going outwards. 

In the above Lemma \ref{PHP trees refinement} we overcome this complication by assuming that $S_{\lambda}$ already takes into account the label of the leaf it is being appended to which is formally expressed as $S_{\lambda}$ being a $\PHPTU{\lambda}$.

When $P'$ is a tree obtained from a tree $P$ by the above process of appending trees to leaves we say that $P'$ \textit{extends} $P$.

\begin{corollary}
\label{PHP extends to uniform PHP}
For every $k' \geq k \in \M$ with $k' \leq n + 1$ a $\PHPTB{\sigma}{\tau}$ of depth $k$ can be extended to a $k'$-uniform $\PHPTB{\sigma}{\tau}$.
\end{corollary}

In the next theorem, we identify $PHP$-trees with the labelings of their leaves or, in other words, with the labelings of their maximal paths.
\begin{theorem}
\label{WPHP array lower bounds}
Let $A$ be an \WAs{m}{k} for $\sigma \in \P$. There exists $k' \geq k$ enjoying $k' \leq |n|^c$ for a suitable standard $c$ and a $k'$-uniform \WAs{m}{k'} of size $N \geq 2m \cdot \frac{(n - |\sigma|)!}{(n - |\sigma| - k')!}$.
\end{theorem}
\begin{proof}
We construct an \WAs{m}{k'} $A'$ in two stages. Firstly, we specify the contents of each row. Then, we specify how to distribute created content among the columns so that the result is an \WAs{m}{k'}.

Let $A_{a}$ be $\bigcup_{b}A(a, b)$ for $a \in [2m]$. We construct $A'_{a}$ (i.e. the contents of the $a$-th row of $A'$) by the following procedure which takes $k$ steps. 

At first, pick $\tau \in A_{a}$. Note that $\tau \cap \sigma = \emptyset$ and $|\tau| \leq k$, implying existence of a $\PHPTB{\sigma}{\tau}$ of depth $\leq 2k$. Extend such a tree to a uniform $\PHPTU{\sigma}$ $P_1$ of depth exactly $2k$ which is achieved utilizing Corollary \ref{PHP extends to uniform PHP}. 

Note that $P_{1}$ satisfies the following properties:
\begin{itemize}
    \item $\forall \rho \in P_{1} \: \exists \pi \in A_{a} \: : \: \pi \| \rho$;
    \item $\forall \rho \in P_{1} \: \forall \pi \in A_{a} \: : \: \pi \| \rho \to |\pi \cap \rho| \geq 1$.
\end{itemize}

The first property follows from the fact that $A$ is an \WAs{m}{k} and any $\rho$ from $P_{1}$ is compatible with $\sigma$. 

The second property is proved by cases. Assuming $\pi = \tau$ it follows that $\pi || \rho \to \pi \geq \rho$ and so $|\pi \cap \rho| = |\pi| \geq 1$. So let $\pi \neq \tau$. In such a case it follows that $\pi \perp \tau$ and so must contain a pair $(a, b)$ such that $\tau$ contains either $(a', b)$ or $(a, b')$ for $a \neq a'$ and $b \neq b'$. In any case a branch $\rho$ of $P_{1}$ compatible with $\pi$ must also contain $(a, b)$, proving $|\pi \cap \rho| \geq |\{(a, b)\}| = 1$. This finishes the first step of the construction.

At the $i$-th step we are given a uniform $\PHPTU{\sigma}$ $P_{i-1}$ of depth $2(i - 1)k$ satisfying the following properties:
\begin{itemize}
    \item $\forall \rho \in P_{i - 1} \: \exists \pi \in A_{a} \: : \: \pi \| \rho$;
    \item $\forall \rho \in P_{i - 1} \: \forall \pi \in A_{a} \: : \: (\pi \| \rho \land |\pi| \geq (i - 1)) \to |\pi \cap \rho| \geq (i - 1)$.
\end{itemize}

Using the first property of $P_{i - 1}$ pick $\pi_{\rho} \in A_{a}$ for each $\rho \in P_{i - 1}$ so that $\pi_{\rho} \| \rho$. For such $\pi_{\rho}$ create a $\PHPTB{\sigma \cup \rho}{\pi_\rho}$ of depth $\leq 2k$ and extend it to a $2k$-uniform $\PHPTU{\sigma \cup \rho}$ $P_i^\rho$. 

Let $P_{i}$ be $P_{i - 1} \triangleleft \{P_{i}^{\rho} \: | \: \rho \in P_{i - 1}\}$. By Lemma \ref{PHP trees refinement} we see that $P_{i}$ is a $2(i \cdot k)$-uniform $\PHPTU{\sigma}$. 

Clearly:
\begin{itemize}
    \item $\forall \rho \in P_{i} \: \exists \pi \in A_{a} \: : \: \pi \| \rho$.
\end{itemize}

To prove the second desired property let $\rho \in P_{i}$ and $\pi \in A_{a}$ so that $\pi \| \rho$ and $|\pi| \geq i$. Let $\rho_{*}$ be the unique maximal path of $P_{i - 1}$ so that $\rho_{*} \geq \rho$. Note that $\pi \| \rho_{*}$ and so $|\pi \cap \rho| \geq |\pi \cap \rho_{*}| \geq (i - 1)$. Assume $|\pi \cap \rho_{*}| = i - 1$. Let $\pi_{\rho_{*}} \in A_{a}$ be the condition chosen to create $P_{i}^{\rho_{*}}$. 

In case $\pi_{\rho_{*}} = \pi$ it follows that $|\pi \cap \rho| = |\pi| \geq i$. So assume $\pi \neq \pi_{\rho_{*}}$. Then, $\pi \perp \pi_{\rho_{*}}$. But since both $\pi$ and $\pi_{\rho_{*}}$ are compatible with $\rho_{*}$ it follows that $(\pi \setminus \rho_{*}) \perp (\pi_{\rho_{*}} \setminus \rho_{*})$ and so $|(\pi \setminus \rho_{*}) \cap (\rho \setminus \rho_{*})| \geq 1$, implying $|\pi \cap \rho| \geq i$.

So, after the end of the $k$-th step we get a $2k^{2}$-uniform $\PHPTU{\sigma}$ $P_k$ satisfying the following properties:
\begin{itemize}
    \item $\forall \rho \in P_{k} \: \exists \pi \in A_{a} \: : \: \pi \| \rho$;
    \item $\forall \rho \in P_{k} \: \forall \pi \in A_{a} \: : \: \pi \| \rho \to \pi \geq \rho$.
\end{itemize}

By Corollary \ref{size of uniform PHP tree conditions} the number of leaves of $P_{k}$ is bounded from below by $\frac{(n - |\sigma|)!}{(n - |\sigma| - 2k^{2})!}$. 

Note that each $\rho \in P_{k}$ is compatible with $\sigma$ and does not intersect it. Furthermore, by Proposition \ref{PHP tree conditions properties} any two different $\rho$ and $\rho'$ from $P_{k}$ are incompatible, and for any $\pi \leq \sigma$ there is a $\rho \in P_{k}$ compatible with it. This implies, assuming the contents of each row of $A'$ are created by the process described above, the resulting $A'$ satisfies all of the properties of uniform \WAs{m}{2k^2}, assuming the content of each row is correctly distributed among the columns to satisfy property 3 of (\ref{WPHP array}).

The distribution proceeds as follows: for any $\rho \in P_{k}$ pick $\tau \in A_{a}$ so that $\rho \leq \tau$. Note that such $\tau$ is unique and is uniquely placed into one of the columns of $A$. Let $b \in [m]$ so that $\tau \in A(a, b)$. Put $\rho$ into $A'(a, b)$. 

The above gives us a well-defined array $A'$ which, as we have already seen, satisfies all of the properties of the (\ref{WPHP array}) up to the third one. 

Let $b, b' \in [m]$ be different and let $\tau \in A'(a, b), \tau' \in A'(a, b')$. Let $\pi$ and $\pi'$ be conditions from $A(a, b)$ and $A(a, b')$, respectively, so that $\tau \leq \pi$ and $\tau' \leq \pi'$. Since $\pi \perp \pi'$ it follows that $\tau \perp \tau'$. This shows that the third property is satisfied, as well.

Finally, by Proposition \ref{size of WPHP array} the size of the array $A'$ is bounded from below by
\begin{align*}
    2m \cdot \min_{a}|\bigcup_{b}A'(a, b)| \geq 2m \cdot \frac{(n - |\sigma|)!}{(n - |\sigma| - 2k^{2})}!
\end{align*}
with the inequality following from the fact that all the rows of $A'$ are $2k^2$-uniform $PHP^{\sigma}$-trees.
\end{proof}

By a slightly more careful construction, one can get $k'$ from the above proof to be equal to $k^{2} + k$ instead of $2k^{2}$. This, however, does not make a major difference.

\begin{theorem}
\label{WPHP array upper bounds}
Let $A$ be a $k$-uniform \WAs{m}{k} for $\sigma \in \P$. Then, the size of $A$ is bounded from above by $m \cdot \frac{(n + 1 - |\sigma|)!}{(n + 1 - |\sigma| - k)!}$.
\end{theorem}
\begin{proof}
The proof is reminiscent of the famous \textit{Erdős–Ko–Rado theorem} as in \cite{katona72}. By Proposition \ref{size of WPHP array} we know that the size of $A$ is bounded from above by $m \cdot \max_{b}|\bigcup_{a}A(a, b)|$.

Let $A^{b}$ denote $\bigcup_{a}A(a, b)$. In particular, each $\tau \in A^{b}$ is of size $k$ and different $\tau, \tau' \in A^{b}$ are incompatible. 

We view each $\tau$ as a $k$-large \textit{matching} of the $K_{n + 1 -  |\sigma|, n - |\sigma|}$ (a \textit{complete bipartite graph} with partitions of sizes $n + 1 - |\sigma|$ and $n - |\sigma|$). We will now prove that for any $k, c, d \in \M$ enjoying $k \leq c \leq d$ it holds that the size of any family $\mathfrak{M}$ of \textit{pairwise-incompatible} $k$-large matchings (i.e. union of any two matchings of $\mathfrak{M}$ is not a matching) of $K_{d, c}$ is bounded from above by $\frac{d!}{(d - k)!}$.

Note that any two different submatchings of some fixed $c$-large matching $M$ of $K_{d, c}$ are compatible and so $\mathfrak{M}$ can contain at most one submatching of such $M$. There are $k! \binom{d}{k} \binom{c}{k}$ different $k$-large matchings of $K_{d, c}$ and each such matching can be extended to $(c - k)! \binom{d - k}{c - k}$ different $c$-large matchings of $K_{d, c}$. Any two $c$-large matchings extending different $k$-large matchings from $\mathfrak{M}$ must necessarily be incompatible and, in particular, different.

It follows that:
\begin{align*}
    |\mathfrak{M}| \leq c! \binom{d}{c} \cdot \frac{1}{(c - k)! \binom{d - k}{c - k}} = \frac{d!}{(d - k)!}.
\end{align*}
\end{proof}

\begin{corollary}
\label{WPHP arrays do not exist}
For any $\sigma \in \P, m \in \M$ and $k \leq |n|^c$ for standard $c$ there does not exist a \WAs{m}{k}. 
\end{corollary}
\begin{proof}
Assume $A$ is a \WAs{m}{k}. By Theorem \ref{WPHP array lower bounds} we can assume it is $k$-uniform and has size $N \geq 2m \cdot \frac{(n - |\sigma|)!}{(n - |\sigma| - k)!}$. 

By Theorem \ref{WPHP array upper bounds} it follows that $N \leq m \cdot \frac{(n + 1 - |\sigma|)!}{(n + 1 - |\sigma| - k)!}$. 

Combining both lower and upper bounds we derive:
\begin{align*}
    2 \leq \frac{n + 1 -|\sigma|}{n + 1 - |\sigma| - k},
\end{align*}
which is a contradiction, since both $|\sigma|$ and $k$ are  bounded by $|n|^{c'}$ for a suitable standard $c$.
\end{proof}

To get the desired contradiction assuming that $(\I, R_G) \nvDash \WPHPD$ we prove the following and key theorem.

\begin{theorem}
\label{WPHP non-standard array exists}
Assume $\emptyset \nVdash \WPHPD$ for $m \in \I$. Then, there exists an \WAs{m}{k} for some $\sigma \in \P$ and $k \leq |n|^c$ for a suitable standard $c$.
\end{theorem}
\begin{proof}

Let $\phi(x, y)$ be $\DR$-formula with parameters from $\I$ so that $\emptyset$ does not force $WPHP_{m}^{2m}(\phi(x, y))$.

Let $D \subseteq \P$ contain such $\rho \in \P$ that one of the below holds:
\begin{itemize}
    \item $\exists a \neq a' \in [2m] \: \exists b \in [m] \: : \: \rho \Vdash \phi(a, b) \land \phi(a', b)$;
    \item $\exists a \in [2m] \: \exists b \neq b' \in [m] \: : \: \rho \Vdash \phi(a, b) \land \phi(a, b')$;
    \item $\exists a \in [2m] \: \forall b \in [m] \: : \: \rho \Vdash \neg \phi(a, b)$.
\end{itemize}
We want to show that $D$ is $\P$-definable. Since $\phi(x, y)$ is a $\DR$-formula and $\DR \subseteq \Sigma^b_1(R)$ we can assume $\phi(x, y)$ is of the form:
\begin{align*}
    \exists \overline{u} \leq t(x, y) \: \psi(x, y, \overline{u}),
\end{align*}
where $t$ is an $L$-term with parameters from $\I$ and $\psi(x, y, \overline{u})$ is sharply bounded. It is now enough to use Theorem \ref{forcing properties} to show $\P$-definability of $D$ by writing it as a union of three sets, one for each condition defining $D$, and then applying different parts of (\ref{forcing properties}) depending on the structure of $\phi(x, y)$.

As an example, we show that the set of all conditions forcing $\phi(a, b)$ for particular $a, b \in \mathbb{I}$ is $\P$-definable. By (\ref{forcing properties}, 7) we get:
\begin{align*}
    \rho \Vdash \exists \overline{u} \leq  t(a, b) \: \psi(a, b, \overline{u})
\end{align*}
iff
\begin{align*}
    \forall \tau \leq \rho \: \exists \pi \leq \tau \: \exists \overline{u} \leq t(a, b) \: : \: \pi \Vdash \psi(a, b, \overline{u}).
\end{align*}
So it is enough to show $\P$-definability of the set of all conditions $\pi$ forcing $\psi(a, b, \overline{u})$. This follows immediately from Proposition \ref{definability forcing}. 

It is clear that any condition in $D$ forces $\WPHPR[m][\phi(x, y)]$. Thus, according to the assumptions, $D$ is not dense. Let $\sigma \in \P$ be a condition not extendable to any condition from $D$.

As a next step, we define an array $A \in \M$ which should satisfy the properties of the \WAs{m}{|n|^c} for a standard $c$. We start by writing the formula $\phi(x, y)$ as $\phi'(x, y, \overline{c})$ so that $\overline{c} \in \I$ and $\phi'(x, y, \overline{z})$ is a $\DR$-formula without parameters.

We shall use the following statement that follows readily
from well-known facts.
\begin{fact}
    There is standard $c \geq 1$ depending just on $\phi'$ and
    parameters $\overline{c}$ and $m$ from $\I$ such that for any $a \in [2m], b \in [m]$ there is a depth $|n|^c$ $\PHPTU{\sigma}$ $T_{a,b} \in \M$ that decides the truth value of $\phi(a,b)$ in $(\I,R_G)$.
\end{fact}

The quickest (although not the most elementary) way to see this 
is to use the already mentioned fact that $\phi'(x, y, \overline{z})$ defines a relation computable by a standard polynomial-time Turing machine $U$ definable in $S^1_2(R)$ and querying oracle about relation $R$ (see \cite[Sections 6 and 7]{krajicek95}). When $\overline{c}
$ and $m$ from $\I$ (i.e. of bit-size polynomial in $|n|$) are fixed and inputs $a,b$ are bounded by fixed $m$, the running time of $U$ is bounded by $|n|^c$ for some standard $c$ and this number bounds also the number of queries $U$ makes on any input. We already know that for generic $G$ the structure $(\I, R_G)$ is a model of $T^1_2(R) \vdash S^1_2(R)$ and so $U$ works there as well.

Given $a,b$, the computation develops according to oracle answers and all possible computations can be represented by a decision tree $T'_{a,b}$ of depth $\leq |n|^c$ that asks about truth values of atomic sentences $R(u,v)$, for some $u \in [n+1]$ and $ v \in [n]$.
 
To represent this by a $\PHPT$ $T_{a,b}$ we simply replace query ``$R(u,v)?$" in $T'_{a,b}$ by a $PHP$-query ``$u \mapsto ?$" (i.e. query asking for a $w$ satisfying $R(u, w)$); if the answer is $v$, then the original query in $T'_{a,b}$ is answered ``yes", in all other cases it is answered ``no". Because we know that the generic $R_G$ violates $\OPHPR$, we can delete from $T_{a,b}$ all paths that are not partial one-to-one maps, getting the desired tree $T_{a,b}$.

As before, we identify $T_{a,b}$ with its maximal paths which, in turn, are represented by partial one-to-one maps between $[n+1]$ and $[n]$ of size $\leq |n|^c$, i.e. as elements of $\P$. Let $T^+_{a,b}$ be all such maps for which the corresponding maximal paths in $T_{a, b}$ lead to acceptance of $\phi(a, b)$.

We define $A$ as $A(a, b) = T^+_{a, b}$. Since $T_{a, b}$ is a $\PHPTU{\sigma}$ two different conditions from $A(a, b)$ are incompatible and all such conditions are compatible with $\sigma$ while their intersections with $\sigma$ are empty.

It is now enough to show that $A$ satisfies all four properties of (\ref{WPHP array}). The first three properties follow immediately from the fact that $\sigma$ is not extendable to a condition which either forces $\phi(a, b) \land \phi(a', b)$ or $\phi(a, b) \land \phi(a, b')$ for $a \neq a'$ and $b \neq b'$. 

Assume the fourth property is not satisfied, i.e.:
\begin{align*}
    \exists \rho \leq \sigma \: \exists a \in [2m] \: \forall b \in [m] \: \forall \tau \in A(a, b) \: : \: \tau \perp \rho.
\end{align*}
It is enough to show $\rho \Vdash \neg \phi(a, b)$ for $a$ as above and all $b \in [m]$ to derive a contradiction. Since $\phi(a, b)$ is $\exists \overline{u} \leq t(a, b) \: \psi(a, b, \overline{u})$ for a sharply bounded $\psi(a, b, \overline{u})$, we get:
\begin{align*}
    \rho \Vdash \neg \phi(a, b) \: \text{ iff } \: \rho \Vdash \forall \overline{u} \leq t(a, b) \: \neg \psi(a, b, \overline{u}),
\end{align*}
where by (\ref{forcing properties}, 4) the latter is equivalent to:
\begin{align*}
    \forall \overline{u} \leq t(a, b) \: \rho \Vdash \neg \psi(a, b, \overline{u}).
\end{align*}

Since $\psi$ is sharply bounded we use (\ref{forcing properties}, 7) to derive:
\begin{align*}
    \rho \Vdash \neg \psi(a, b, \overline{u}) \: \text{ iff } \: \forall \pi \leq \rho \: \pi \nVdash \psi(a, b, \overline{u}).
\end{align*}

Pick any $\pi$ extending $\rho$ and assume it forces $\psi(a, b, \overline{u})$. Thus, $\pi \Vdash \phi(a, b)$ which implies $\pi$ is compatible with some $\tau \in A(a, b)$. However, this implies $\rho \| \tau$ since $\pi \leq \rho$ contradicting our initial assumption that $\rho$ is incompatible with every $\tau$ from $A(a, b)$ for every $b \in [m]$.

\end{proof}

\begin{corollary}
\label{empty forces WPHP}
\begin{align*}
    \emptyset \Vdash \WPHPD,
\end{align*}
implying that for any generic filter $G$ it holds that:
\begin{align*}
    (\I, R_{G}) \vDash \TR + \forall m \WPHPD + \exists n \neg \OPHPR.
\end{align*}
\end{corollary}

As an application of our proof methods, we derive the following result already advertised in the abstract.

\begin{theorem}
\label{ajtai}
There exists a model of:
\begin{align*}
    T_{2}^{1}(R) + \exists n (\forall m \leq n^{1 - \epsilon} (PHP_{m}^{m + 1}(\DR)) + \neg PHP_{n}^{n + 1}(R)),
\end{align*}
where $\epsilon$ is an arbitrarily small standard rational.
\end{theorem}
\begin{proof}
The model is actually the same $(\I, R_G)$ as in \ref{empty forces WPHP}. 

Assuming: 
\begin{align*}
    \emptyset \nVdash \forall m \leq n^{1 - \epsilon} (PHP_{m}^{m + 1}(\DR))   
\end{align*}
it follows that one can construct $(n^{1 - \epsilon}, |n|^{c}, \sigma)$-$PHP$-$array$ defined analogously as in \ref{WPHP array} with the only difference being that the indexing set is $[n^{1 - \epsilon} + 1] \times [n^{1 - \epsilon}]$.

Theorems \ref{WPHP array lower bounds} and \ref{WPHP array upper bounds} still apply implying:
\begin{align*}
    (n^{1 - \epsilon} + 1) \cdot \frac{(n - |\sigma|)!}{(n - |\sigma| - |n|^{c})!} \leq n^{1 - \epsilon} \cdot \frac{(n + 1 - |\sigma|)!}{(n + 1 - |\sigma| - |n|^{c})!},
\end{align*}

The above results in:
\begin{align*}
    \frac{n^{1 - \epsilon} + 1 }{n^{1 - \epsilon}} \leq \frac{n + 1 - |\sigma|}{n + 1 - |\sigma| - |n|^{c}},
\end{align*}
which is a contradiction, since $|\sigma| \leq |n|^{c'}$ and $c', c, \epsilon$ are standard.
\end{proof}

\section*{Acknowledgements}
The work first appeared in \cite{narusevych22} supervised by Jan Krajíček. I am indebted to my supervisor for detailed comments and suggestions.

\bibliographystyle{plain}
\bibliography{main}

\end{document}